\documentclass[10pt]{amsart}
\usepackage[T1]{fontenc}
\usepackage[english]{babel}
\usepackage{graphicx}                         
\usepackage{amsmath}                    
\usepackage{amssymb} 
\usepackage{mathrsfs}
\usepackage{tikz}

\usepackage{tikz-cd}

\usepackage[margin=1.3in]{geometry}
\usepackage{color}
\usepackage{hyperref}
\usepackage[noabbrev,capitalize]{cleveref}
\hypersetup{
	colorlinks,
	linkcolor={red!50!black},
	citecolor={blue!50!black},
	urlcolor={blue!80!black}
}
\usepackage{palatino}

\newtheorem{thm}{Theorem}

\DeclareMathOperator{\Hom}{Hom}
\title{A closer look at Kadeishvili's theorem}

\author{Dan Petersen}
\address{Matematiska institutionen \\ Stockholms Universitet\\ 106 91 Stockholm \\ Sweden}
\email{dan.petersen@math.su.se}

\begin{document}

\maketitle

\begin{abstract}We give a proof of the Homotopy Transfer Theorem following Kadeishvili's original strategy. Although Kadeishvili originally restricted himself to transferring a dg algebra structure to an $A_\infty$-structure on homology, we will see that a small modification of his argument proves the general case of transferring any kind of $\infty$-algebra structure along a quasi-isomorphism, under weaker hypotheses than existing proofs of this result. 
\end{abstract}




\section{Introduction}

In 1980, Tornike Kadeishvili published the following celebrated theorem \cite{kadeishvili}:

\begin{thm}[Homotopy Transfer Theorem]
	Let $A$ be a dg algebra over a commutative ring $R$. Assume that the homology $H(A)$ is a projective $R$-module, so that there exists a quasi-isomorphism $f \colon H(A)\to A$. There exists noncanonically an $A_\infty$-algebra structure on $H(A)$ with vanishing differential, and an $A_\infty$-quasi-isomorphism $H(A) \to A$ whose arity $1$ term is given by $f$. 
\end{thm}

A very large number of different proofs of the Homotopy Transfer Theorem have been given since, in various strengthened forms. Most of these arguments use either the Homological Perturbation Lemma (e.g.\ \cite{gugenheimlambe,gugenheimlambestasheff,huebschmann-kadeishvili,huebschmann-stasheff,manetti-relative}, and in the most general formulation, \cite{berglundhomologicalperturbation}), or a formulation in terms of sums over trees due to Kontsevich--Soibelman (see \cite{kontsevich-soibelman,merkulov-strong,vanderlaanthesis}, and \cite[Section 10.3]{lodayvallette} for a very general textbook treatment) but other completely different methods of proof are also possible \cite{bergermoerdijk,marklhomotopyalgebras,rogershomotopical,preliedeformationtheory}. 
In more general versions of this theorem one can allow $A$ itself to be an $A_\infty$-algebra to begin with, the map $H(A)\to A$ can be replaced by a quasi-isomorphism from some chain complex to $A$, and rather than $A_\infty$-algebras one may consider $\infty$-algebras over other operads. On the other hand all these other versions require stronger hypotheses on $R$ or the complexes involved. 

Each of the proofs mentioned above requires setting up some amount of general machinery. By contrast, Kadeishvili's argument is as direct as it could be: he writes down the infinite list of equations describing an $A_\infty$-structure on $H(A)$ and an $A_\infty$-morphism $H(A)\to A$, and argues inductively that each of the equations can be solved in turn. 
The goal of this note is to give a direct proof of the Homotopy Transfer Theorem following Kadeishvili's original approach. Somewhat surprisingly, it turns out that once Kadeishvili's argument has been written down in the right way it works for all sorts of $\infty$-algebras, and one can transfer the algebraic structure over a general quasi-isomorphism. In fact the resulting version of the Homotopy Transfer Theorem requires weaker assumptions than any statement that I am aware of in the literature. (Its drawback compared to arguments by sums over trees or homological perturbation theory is of course the nonconstructive nature: one does not get an explicit formula for the transferred structure.) That being said, I do not claim any great originality in this result. 

Before stating the theorem we will need some notation. Let $R$ be any commutative ring, and let $C$ be a conilpotent cooperad in graded $R$-modules satisfying $C(0)=0$ and $C(1) \cong R$. We denote the cofree conilpotent $C$-coalgebra cogenerated by a graded $R$-module $V$ by $C(V) = \bigoplus_{n\geq 1} (C(n)\otimes V^{\otimes n})^{\mathbb S_n}$. Any coderivation of $C(V)$ is uniquely determined by a linear map $C(V) \to V$, and by the \emph{arity $1$ term} of a coderivation we mean the component $V = C(1) \otimes V \to V$. Similarly a $C$-coalgebra morphism $C(V) \to C(W)$ is described by a linear map $C(V) \to W$, and by the arity $1$ term of such a morphism we mean the component $V \to W$. If $M$ and $N$ are dg $R$-modules, then we let $\underline \Hom_R(M,N)$ denote the ``internal Hom'' dg $R$-module of homomorphisms $M \to N$. 

\begin{thm}\label{thmA}
	Let $(V,d_V)$ and $(W,d_W)$ be dg $R$-modules, and $f \colon V \to W$ a chain map. Let $\nu$ be a square-zero coderivation  of $C(W)$ of degree $-1$ whose arity $1$ term equals the given differential $d_W$. Assume that $f$ induces a quasi-isomorphism $\underline \Hom_R(C(V),V) \to \underline{\Hom}_R(C(V),W)$. Then there exists noncanonically a square-zero coderivation $\mu$ of $C(V)$ whose  arity $1$ term is $d_V$, and a morphism of $C$-coalgebras $F \colon C(V) \to C(W)$ whose  arity $1$ term is $f$ and which is a chain map with respect to the differentials defined by $\mu$ and $\nu$. 
\end{thm}

In particular, we get a Homotopy Transfer Theorem if we assume that $f$ is a quasi-isomorphism and that $C(V)$ is $K$-projective in the sense of Spaltenstein \cite{spaltenstein} (that is, the functor $\underline{\Hom}_R(C(V),-)$ preserves exactness; for example, bounded below complexes of projective modules are $K$-projective). For $R=\mathbf Z$, the complex $C(V)$ is $K$-projective e.g.\ if each $C(n)$ is a free abelian group, and $V$ is a bounded below chain complex of free abelian groups. Note that we do not need to assume e.g.\ that $C$ is $\Sigma$-cofibrant, or that $f$ has a quasi-inverse. For example, $V$ might be a projective resolution of $W$; the fact that one can transfer an $A_\infty$-structure to a projective resolution is a rather recent theorem of Burke \cite{burkeprojective}.

\cref{thmA} specializes to Kadeishvili's original result since the structure of an $A_\infty$-algebra on a graded $R$-module $V$ is equivalent to a square zero coderivation of the reduced tensor coalgebra of the suspension of $V$. More generally if $P$ is a Koszul operad and $C$ is its Koszul dual cooperad, then a square zero coderivation of $C(V)$ is the same thing as a $P_\infty$-algebra structure on $V$, and a $P_\infty$-morphism $V \to W$ is a morphism $C(V) \to C(W)$ which is a chain map with respect to the differentials defined by the coderivations. So we recover the Homotopy Transfer Theorem for $\infty$-algebras over a Koszul operad. For example, $L_\infty$-algebras and $C_\infty$-algebras are described as coderivations of cofree cocommutative coalgebras and cofree Lie coalgebras, respectively. See \cite[Chapter 10]{lodayvallette}. 

By minor modifications of the argument one also obtains a Homotopy Transfer Theorem in the ``other direction'', as well as proofs of uniqueness of the transferred structures. More precisely we have:

\begin{thm}\label{thmB}
	Let $(V,d_V)$ and $(W,d_W)$ be dg $R$-modules, and $f \colon V \to W$ a chain map. Let $\mu$ be a square-zero coderivation  of $C(V)$ of degree $-1$ whose arity $1$ term equals the given differential $d_V$. Assume that $f$ induces a quasi-isomorphism $\underline \Hom_R(C(W),W) \to \underline{\Hom}_R(C(V),W)$. Then there exists noncanonically a square-zero coderivation $\nu$ of $C(W)$ whose  arity $1$ term is $d_W$, and a morphism of $C$-coalgebras $F \colon C(V) \to C(W)$ whose  arity $1$ term is $f$ and which is a chain map with respect to the differentials defined by $\mu$ and $\nu$. 
	
\end{thm}

\begin{thm}\label{thmC}
	Keep the hypotheses and notation of \cref{thmA}. Let $\mu$ and $\mu'$ be square zero coderivations of $C(V)$ obtained as in \cref{thmA}, and let $F, F' : C(V) \to C(W)$ be the corresponding morphisms. There exists a noncanonical isomorphism of $C$-coalgebras $\Phi \colon C(V) \to C(V)$ whose  arity $1$ term is the identity and which is a chain map with respect to the differentials defined by $\mu$ and $\mu'$. Moreover, if the cooperad $C$ is nonsymmetric then there exists a coderivation homotopy $H$ between $F' \circ \Phi$ and $F$. 
\end{thm}

\begin{thm}\label{thmD}
	Keep the hypotheses and notation of \cref{thmB}. Let $\nu$ and $\nu'$ be square zero coderivations of $C(W)$ obtained as in \cref{thmB}, and let $F, F' : C(V) \to C(W)$ be the corresponding morphisms. There exists a noncanonical isomorphism of $C$-coalgebras $\Psi \colon C(W) \to C(W)$ whose  arity $1$ term is the identity and which is a chain map with respect to the differentials defined by $\nu$ and $\nu'$. Moreover, if the cooperad $C$ is nonsymmetric then there exists a coderivation homotopy $H$ between $\Psi \circ F' $ and $F$. 
\end{thm}

\addtocounter{thm}{-4}

\section{Recollections on coalgebras and coderivations}

Let $C$ be a conilpotent cooperad in dg $R$-modules. In the main results of the paper we will assume that $C(0)=0$ and $C(1) \cong R$; in particular, $C$-coalgebras generally do not have a counit. For any dg $R$-module $V$ and any $n \geq 0$ we set $C^n(V) = (C(n) \otimes V^{\otimes n})^{\mathbb S_n}$, where $(-)^{\mathbb S_n}$ denotes $\mathbb S_n$-invariants, and denote by
$$ C(V) = \bigoplus_{n \geq 0} C^n(V) $$
the cofree conilpotent $C$-coalgebra cogenerated by $V$. We also write $C^{\leq N}(V) = \bigoplus_{n =0}^NC^n(V)$, so that $C(V)$ is the increasing union of subcoalgebras $C^{\leq n}(V)$. Readers who are only interested in the $A_\infty$-case of the main theorem may restrict their attention to the case that $C(V)$ is the reduced tensor coalgebra on $V$. 

\subsection{Arity decomposition of coalgebra morphisms}

Let $E$ be a conilpotent dg $C$-coalgebra. The fact that $C(V)$ is cofree means precisely that coalgebra homomorphisms $E \to C(V)$ are in natural bijection with $R$-linear maps $E \to V$. In particular, this means that coalgebra homomorphisms between cofree coalgebras $C(V) \to C(W)$ are given by elements of
$$ \Hom_R(C(V),W) \cong \prod_{n\geq 0}\Hom_R(C^n(V),W).$$
If $F \colon C(V) \to C(W)$ is a coalgebra morphism then we denote by $F_{(n)} \colon C^n(V) \to W$ the corresponding factor in the above decomposition, and we call it the \emph{arity $n$ term} of $F$.

Any chain map $C(V) \to C(W)$ can be written as a sum of maps $C^n(V) \to C^k(W)$. If $F \colon C(V) \to C(W)$ is a coalgebra morphism with arity terms $F_{(i)}$, then the component $C^n(V) \to C^k(W)$ of $F$ can be described as follows: for any composition of the integer $n$ into $k$ summands, $n_1+\ldots+n_k$, we get a map 
\[\begin{tikzcd}
C^n(V) \arrow[r,"\textrm{cocomposition}"] &[30pt] C(k)  \otimes C^{n_1}(V) \otimes \ldots \otimes C^{n_k}(V) \arrow[r,"\mathrm{id}\otimes F_{(n_1)} \otimes \ldots \otimes F_{(n_k)}"]&[50pt] C(k) \otimes W^{\otimes k}.
\end{tikzcd}\]
If we sum these maps over all compositions of the integer $n$ into $k$ summands, then we land in the $\mathbb S_k$-invariants $(C(k) \otimes W^{\otimes k})^{\mathbb S_k}$, defining a chain map $C^n(V)\to C^k(W)$.

Using the previous paragraph we see in particular how to describe the arity terms of a composition $G \circ F$ of two coalgebra morphisms, $F \colon C(V) \to C(W)$ and $G \colon C(W) \to C(U)$, in terms of the arity terms of $F$ and $G$. Indeed, $(G \circ F)_{(n)}$ is given by the sum of all of the maps
$$ C^n(V) \to C^k(W) \to U$$
for $k=1,\ldots, n$, where the first arrow is described in the previous paragraph and the second is the arity $k$ term of $G$. 

\subsection{Arity decomposition of coderivations}
Let $E$ be a conilpotent $C$-coalgebra and $M$ an $E$-comodule. A map of dg $R$-modules $M \to E$ is called a \emph{coderivation} if it satisfies the co-Leibniz rule, meaning that the diagram
\[\begin{tikzcd}
	M \arrow[r] \arrow[d]& E \arrow[d]\\ \bigoplus_{k=0}^{n-1}
C(n) \otimes E^{\otimes k} \otimes M \otimes E^{\otimes (n-k-1)} \arrow[r]& C(n) \otimes E^{\otimes n} 
	\end{tikzcd}
\]
commutes for all $n$. 

The composition of two coderivations is not generally a coderivation, but if $\mu$ is a coderivation of odd degree then its square $\mu \circ \mu$ is also a coderivation. To see this one may use the co-Leibniz identity
$$ \Delta_n \circ \mu = \sum_{k=0}^{n-1} (\mathrm{id}_{C(n)} \otimes \mathrm{id}^{\otimes k} \otimes \mu \otimes \mathrm{id}^{\otimes (n-k-1)} ) \circ \Delta_n $$ 
to expand the composition $\Delta_n \circ \mu \circ \mu$; doing this, one ends up with $n$ terms of the form $$(\mathrm{id}_{C(n)} \otimes \mathrm{id}^{\otimes k} \otimes (\mu \circ \mu) \otimes \mathrm{id}^{\otimes (n-k-1)} ) \circ \Delta_n$$ 
and all remaining terms will be of the form $(\mathrm{id}_{C(n)} \otimes \mathrm{id}^{\otimes k} \otimes \mu \otimes \mathrm{id}^{\otimes l} \otimes \mu \otimes \mathrm{id}^{\otimes (n-k-l-1)} ) \circ \Delta_n$; however, all terms of the latter form appear twice with opposite signs because of the Koszul sign rule, since $\mu$ was of odd degree. A related fact, proved in a very similar way, is that the commutator of two coderivations $E \to E$ is a coderivation: $\mathrm{Coder}_R(E,E)$ is a Lie subalgebra of $\Hom_R(E,E)$. 

The pre- or postcomposition of a morphism of coalgebras with a coderivation is again a coderivation. 

Coderivations into the cofree coalgebra $C(V)$ are in natural bijection with linear maps into $V$, in the same way as we could describe coalgebra homomorphisms into cofree coalgebras. Specifically, we have for any $C(V)$-comodule $M$ a natural bijection
$$ \mathrm{Coder}_R(M,C(V)) \cong \Hom_R(M,V).$$
In particular, if we are given a coalgebra morphism $F \colon C(V) \to C(W)$, so that $C(V)$ is a $C(W)$-comodule, then coderivations from $C(V)$ to $C(W)$ are in bijection with elements of
$$ \prod_{n \geq 0} \Hom_R(C^n(V),W).$$
If $\eta \colon C(V) \to C(W)$ is a coderivation, then we denote by $\eta_{(n)} \colon C^n(V) \to W$ the $n$th factor of the above decomposition, and we call it the \emph{arity $n$ term} of $\eta$. 

Let $\eta \colon C(V) \to C(W)$ be a coderivation with respect to a coalgebra morphism $F \colon C(V) \to C(W)$. We can write $\eta$ as a sum of maps $C^n(V) \to C^k(W)$; let us express these maps in terms of the arity terms of $\eta$ and $F$, just like we did for coalgebra morphisms. Then $C^n(V) \to C^k(W)$ is given by the sum of the maps
\[\begin{tikzcd}
C^n(V) \arrow[r] & C(k)  \otimes C^{n_1}(V) \otimes \ldots \otimes C^{n_k}(V) \arrow[r,"\mathrm{id}\otimes F_{(n_1)} \otimes \ldots \otimes \eta_{(n_i)} \otimes \ldots \otimes F_{(n_k)}"]&[90pt] C(k) \otimes W^{\otimes k}
\end{tikzcd}\]
over all integer compositions $n=n_1 + \ldots + n_{i-1} + \boldsymbol{n_i} + n_{i+1} + \ldots + n_k$ with a distinguished summand. In particular, if $F$ is the identity map $C(V) \to C(V)$, then the component $C^n(V) \to C^k(V)$ of $\eta$ is given by summing up the maps
\[\begin{tikzcd}
C^n(V) \arrow[r] & C(k)  \otimes V^{\otimes (i-1)} \otimes C^{n-k+1}(V) \otimes V^{\otimes (k-i)} \arrow[r,"\eta_{(n-k+1)}"]&[15pt] C(k) \otimes V^{\otimes k}.
\end{tikzcd}\]
for $i=1,\ldots,k$. 

In particular, the above formula can be used to express the arity terms of a composition of a coderivation and a coalgebra morphism, or the commutator of two coderivations, in terms of their individual arity terms, exactly like we did previously for compositions of coalgebra morphisms. In the paper we will only apply this in the simplest case, which is when we have a coderivation $\eta$ which vanishes on $C^{\leq (n-1)}(V)$ and we want to calculate the arity $n$ term of a composition involving $\eta$; in this case, nearly all terms in the above sums will vanish. 

\section{Proof of the Homotopy Transfer Theorem}

We denote by $\underline{\Hom}_R$ the inner Hom of dg $R$-modules: if $(M,d_M)$ and $(N,d_N)$ are dg $R$-modules, then 
$$\underline{\Hom}_R(M,N)_n = \prod_{k \in \mathbf Z}\Hom_R(M_k,N_{n+k}),$$
with differential defined by $\partial f = d_N \circ f - (-1)^n f \circ d_M$ for $f \in \underline{\Hom}_R(M,N)_n$.

We fix a cooperad $C$ in dg $R$-modules satisfying $C(0)=0$ and $C(1) \cong R$. If $V$ is a dg $R$-module, then $C(V)$ is naturally differential graded; the induced differential of $C(V)$ is a coderivation, and its only nonzero arity term is the arity $1$ term. In the following proof we will repeatedly consider the Hom-complexes $\underline{\Hom}_R(C^n(V),W)$, where $(V,d_V)$ and $(W,d_W)$ are dg $R$-modules; the differential on $C^n(V)$ is induced by the differentials on $C(n)$ and on $V$, and we always use $\partial$ to denote the above differential on $\underline{\Hom}_R(C^n(V),W)$.

\begin{thm}\label{thmAagain}
	Let $(V,d_V)$ and $(W,d_W)$ be dg $R$-modules, and $f \colon V \to W$ a chain map. Let $\nu$ be a square-zero coderivation  of $C(W)$ of degree $-1$ whose arity $1$ term equals the given differential $d_W$. Assume that $f$ induces a quasi-isomorphism $\underline \Hom_R(C(V),V) \to \underline{\Hom}_R(C(V),W)$. Then there exists noncanonically a square-zero coderivation $\mu$ of $C(V)$ whose  arity $1$ term is $d_V$, and a morphism of $C$-coalgebras $F \colon C(V) \to C(W)$ whose  arity $1$ term is $f$ and which is a chain map with respect to the differentials defined by $\mu$ and $\nu$. 
\end{thm}

\begin{proof}
Let $n \geq 2$, and suppose we are given a degree $-1$ coderivation $\mu \colon C(V) \to C(V)$ with $\mu_{(1)} = d_V$, and a coalgebra morphism $F \colon C(V) \to C(W)$ with $F_{(1)} = f$, such that the restrictions of $\mu$ and $F$ to $C^{\leq (n-1)}(V)$ satisfy 
$$ \begin{cases} \mu \circ \mu =0 \\ F \circ \mu - \nu \circ F = 0. \end{cases}$$
We will prove that by modifying the arity $n$ terms of $\mu$ and $F$ we can arrange so that these equations are satisfied also on $C^{\leq n}(V)$. This will finish the proof by induction on $n$, since both equations are clearly satisfied on $C^{\leq 1}(V)=V$. 

Note that $F \circ \mu - \nu \circ F$ is a coderivation, and so is $\mu \circ \mu$. Now the following equations are obviously satisfied:
\begin{equation}\label{1}
\mu \circ (\mu \circ \mu) - (\mu \circ \mu) \circ \mu = 0
\end{equation}
and
\begin{equation}\label{2}
F \circ (\mu \circ \mu) = (F \circ \mu - \nu \circ F) \circ \mu + \nu \circ (F \circ \mu - \nu \circ F).
\end{equation}
Let us compute the arity $n$ term of both these equations. For \eqref{1}, we use that terms of $\mu \circ \mu$ of arity less than $n$ vanish. Hence the arity $n$ term of the left hand side of \eqref{1} is given by a sum over all ways of precomposing and postcomposing $(\mu \circ \mu)_{(n)}$ with $\mu_{(1)}$, as all other terms in this arity vanish, so that we obtain the identity
\begin{equation}
\partial (\mu \circ \mu)_{(n)} = 0 \qquad \text{in  }  \underline{\Hom}_R(C^{n}(V),V). \label{3}
\end{equation}
Similarly we may consider the arity $n$ term of \eqref{2}, and use that all terms of both  $\mu \circ \mu$  and $(F \circ \mu - \nu \circ F)$ of arity below $n$ vanish. We obtain from \eqref{2} the identity 
\begin{equation}\label{4} f \circ (\mu \circ \mu)_{(n)} = \partial  (F \circ \mu - \nu \circ F)_{(n)} \qquad \text{in  }  \underline \Hom_R(C^{n}(V),W),\end{equation}
since the differential in $\underline \Hom_R(C^{n}(V),W)$ is the sum over all ways of precomposing and postcomposing with $\mu_{(1)}$ and $\nu_{(1)}$, respectively. 

Recall that $f$ induces a quasi-isomorphism of chain complexes
$$ \underline \Hom_R(C^n(V),V) \to \underline \Hom_R(C^ n(V),W).$$ Now $(\mu \circ \mu)_{(n)}$ is a cycle by \eqref{3}, and it is mapped under $f$ to a boundary by \eqref{4}. It follows that $(\mu \circ \mu)_{(n)}$ must itself be a boundary in $\underline \Hom_R(C^{n}(V),V)$; say that there exists $e \colon C^{n}(V) \to V$ of degree $-1$ such that $\partial e = (\mu \circ \mu)_{(n)}$. Let $\mu'$ be the coderivation of $C(V)$ which has the same arity terms as $\mu$ except $\mu_{(n)}' = \mu_{(n)}-e$.  Then
\begin{equation}
(\mu ' \circ \mu')_{(n)} = 0, \label{5}
\end{equation} 
i.e.\ $\mu' \circ \mu'=0$ in $C^{\leq n}(V)$. Moreover we then have by \eqref{4} that $(F \circ \mu' - \nu \circ F)$ is a cycle in $\underline\Hom_R(C^{n}(V),W)$, which (again since $ \underline\Hom_R(C^n(V),V) \to \underline\Hom_R(C^n(V),W)$ is a quasi-isomorphism) means that it can be written as the sum of the image of a cycle under $f$, and a boundary. Thus we choose $e' \in \underline\Hom_R(C^{n}(V),V)$ with $\partial e'=0$, and $e'' \in \underline\Hom_R(C^{n}(V),W)$, such that 
\begin{equation}
(F \circ \mu' - \nu \circ F )_{(n)} = f \circ e' + \partial e''. \label{6}
\end{equation}
Let $\mu''$ be the coderivation of $C(V)$ which has the same arity terms as $\mu'$ except $\mu_{(n)}'' = \mu_{(n)}'-e'$, and define similary a new morphism $F'$ with $F'_{(n)}=F_{(n)}- e''$. Since we only modified $\mu'$ by adding a cycle, we can still argue as in \eqref{5} to see that $\mu'' \circ \mu'' = 0$ in $C^{\leq n}(V)$. Moreover, it is straightforward to check that \eqref{6} says exactly that $(F' \circ \mu'' - \nu \circ F' )_{(n)} = 0$. The theorem is proven.\end{proof}

The preceding proof is essentially Kadeishvili's. Let us make the comparison explicit. At one point of the argument Kadeishvili writes ``Direct calculations show that $\partial U_n=0$'', and later that ``The remaining condition (1) can be proved by a straightforward check''. Kadeishvili's $\partial U_n$ is our $\partial (F \circ \mu - \nu \circ F)_{(n)}$, and his condition (1) is our condition $ (\mu \circ \mu)_{(n)} = 0$. As in the above argument a calculation shows that $f \circ (\mu \circ \mu)_{(n)} = \partial (F \circ \mu - \nu \circ F)_{(n)}$. For Kadeishvili the map $f$ is a cycle-choosing homomorphism, so the fact that the value of $f$ is a boundary implies both that $(\mu \circ \mu)_{(n)}=0$ and $(F \circ \mu - \nu \circ F)_{(n)} = 0$. 

A final comment is that the condition that $\underline \Hom_R(C(V),V) \to \underline{\Hom}_R(C(V),W)$ is a quasi-isomorphism is forced upon us very naturally by the structure of the inductive strategy. Firstly we find a cycle in $\underline \Hom_R(C(V),V)$ whose image in $\underline{\Hom}_R(C(V),W)$ is a boundary, and what we need to continue the process is that the cycle is a boundary already in $\underline \Hom_R(C(V),V)$, which means precisely that the map is injective on homology. Secondly we find a class in $\underline{\Hom}_R(C(V),W)$, and what we need to continue the process is that the cycle can be written as the sum of a boundary and a cycle that is in the image, which means precisely that the map is surjective on homology. Some versions of homotopy transfer in the literature only require a one-sided inverse to the map we transfer along (see e.g.\ \cite{markltransferring}), which may suggest that it would suffice to assume that the induced map on homology is injective or surjective, but I do not see a natural way to carry out the strategy under any such milder assumption. The group homomorphism $\mathbf Z/2 \to \mathbf Z/4$ can not be turned into an algebra homomorphism for any ring structure on $\mathbf Z/4$, which shows at least that the main result of \cite{markltransferring} does not hold integrally. 

\section{Variations}

Let us briefly explain the necessary modifications of the argument to obtain  the transfer in the other direction, and the uniqueness of the transferred structure.

\begin{thm}
	Let $(V,d_V)$ and $(W,d_W)$ be dg $R$-modules, and $f \colon V \to W$ a chain map. Let $\mu$ be a square-zero coderivation  of $C(V)$ of degree $-1$ whose arity $1$ term equals the given differential $d_V$. Assume that $f$ induces a quasi-isomorphism $\underline \Hom_R(C(W),W) \to \underline{\Hom}_R(C(V),W)$. Then there exists noncanonically a square-zero coderivation $\nu$ of $C(W)$ whose  arity $1$ term is $d_W$, and a morphism of $C$-coalgebras $F \colon C(V) \to C(W)$ whose  arity $1$ term is $f$ and which is a chain map with respect to the differentials defined by $\mu$ and $\nu$. 
	
\end{thm}

\begin{proof}The structure of the argument is the same as in the proof of \cref{thmAagain}. We suppose instead that we have a degree $-1$ coderivation $\nu \colon C(W) \to C(W)$ with $\nu_{(1)} = d_W$, and a morphism $F \colon C(V) \to C(W)$ with $F_{(1)} = f$, such that the restrictions of $\nu$ and $F$ to $C^{\leq (n-1)}(W)$, resp.\ $C^{\leq (n-1)}(V)$, satisfy 
	$$ \begin{cases} \nu \circ \nu =0 \\ F \circ \mu - \nu \circ F = 0. \end{cases}$$
We now consider the two equations
	\begin{equation*}\label{11}
	\nu \circ (\nu \circ \nu) - (\nu \circ \nu) \circ \nu = 0
	\end{equation*}
	and
	\begin{equation*}\label{21}
	(\nu \circ \nu) \circ F = (F \circ \mu - \nu \circ F) \circ \mu + \nu \circ (F \circ \mu - \nu \circ F)
	\end{equation*}
	and compute the arity $n$ term of both these equations. By the same argument as before we obtain
	\begin{equation*}
	\partial (\nu \circ \nu)_{(n)} = 0 \qquad \text{in  }  \underline{\Hom}_R(C^{n}(W),W), 
	\end{equation*}and the identity 
	\begin{equation*}(\nu \circ \nu)_{(n)} \circ C^n(f) = \partial  (F \circ \mu - \nu \circ F)_{(n)} \qquad \text{in  }  \underline \Hom_R(C^{n}(V),W),\end{equation*}
	where $C^n(f)$ denotes the map $(C(n) \otimes V^{\otimes n})^{\mathbb S_n} \stackrel{\mathrm{id} \otimes f^{\otimes n}} \longrightarrow (C(n) \otimes W^{\otimes n})^{\mathbb S_n}$. By the same argument as before it follows that $(\nu \circ \nu)_n$ is a boundary in $\underline \Hom_R(C^{n}(W),W)$; say that there exists $e \colon C^{n}(W) \to W$ such that $\partial e = (\nu \circ \nu)_{(n)}$. Let $\nu'$ be the coderivation of $C(W)$  for which $\nu_{(n)}' = \nu_{(n)}-e$.  Then
	\begin{equation*}
	(\nu ' \circ \nu')_{(n)} = 0,
	\end{equation*} 
	and it follows that $(F \circ \mu - \nu' \circ F)$ is a cycle in $\underline\Hom_R(C^{n}(V),W)$, which again means that it can be written as the sum of the image of a cycle under $f$, and a boundary. Thus we choose $e' \in \underline\Hom_R(C^{n}(W),W)$ with $\partial e'=0$, and $e'' \in \underline\Hom_R(C^{n}(V),W)$, such that 
	\begin{equation*}
	(F \circ \mu - \nu' \circ F )_{(n)} = e' \circ C^n(f) + \partial e''.
	\end{equation*}
	Let $\nu''$ be the coderivation of $C(W)$ with $\nu_{(n)}'' = \nu_{(n)}'-e'$, and let similarly $F'$ be the coalgebra morphism with $F'_{(n)}=F_{(n)}- e''$. By the same argument as before we see that $(F' \circ \mu - \nu'' \circ F' )_{(n)} = 0$ as claimed. 
\end{proof}

Suppose that $C$ is a nonsymmetric cooperad, and let $F, F' \colon E \to D$ be morphisms of $C$-coalgebras. An \emph{$(F,F')$-coderivation} is a map $H \colon E \to D$ making the following diagram commute for all $n$:
\[\begin{tikzcd}[column sep=6cm]
E \arrow[r,"H"] \arrow[d]& D \arrow[d]\\
C(n) \otimes E^{\otimes n}  \arrow[r,"\sum_{k=0}^{n-1} \mathrm{id}_{C(n)} \otimes F^{\otimes k} \otimes H \otimes (F')^{\otimes (n-k+1)}"]& C(n) \otimes D^{\otimes n} .
\end{tikzcd}
\]
An $(F,F')$-coderivation $E \to C(W)$ is completely determined by its composition with the projection onto the cogenerators $C(W) \to W$, just as for usual coderivations. This sets up a natural bijection between $(F,F')$-coderivations $E \to C(W)$ and $R$-linear maps $E \to W$. If $H \colon C(V) \to C(W)$ is an $(F,F')$-coderivation then we write $H_{(n)}$ for the corresponding map $C^n(V) \to W$, the ``arity $n$ term'' of $H$. We say that an $(F,F')$-coderivation $H$ is a \emph{coderivation homotopy} between $F$ and $F'$ if $\partial H = F - F'$ in $\underline{\Hom}_R(E,D)$.
When $C$ is the coassociative cooperad, this specializes to the usual notion of a homotopy between two $A_\infty$-morphisms, and it can be interpreted as the homotopy relation on morphisms defined in terms of a cylinder functor in a suitable model category of coalgebras \cite[Section 1.3.4]{lefevre-hasegawa}. 

If $C$ is a symmetric cooperad the above definition of $(F,F')$-coderivation still makes sense but is not useful; the image of the left vertical arrow in the diagram lands in the $\mathbb S_n$-invariants and the lower horizontal arrow is very much not $\mathbb S_n$-invariant, so nontrivial $(F,F')$-coderivations will generally not exist. 

\begin{thm}
	Keep the hypotheses and notation of \cref{thmA}. Let $\mu$ and $\mu'$ be square zero coderivations of $C(V)$ obtained as in \cref{thmA}, and let $F, F' : C(V) \to C(W)$ be the corresponding morphisms. There exists a noncanonical isomorphism of $C$-coalgebras $\Phi \colon C(V) \to C(V)$ whose  arity $1$ term is the identity and which is a chain map with respect to the differentials defined by $\mu$ and $\mu'$. Moreover, if the cooperad $C$ is nonsymmetric then there exists a coderivation homotopy between $F' \circ \Phi$ and $F$. 
\end{thm}

\begin{proof}Let us focus on the case that $C$ is nonsymmetric. Let $n \geq 2$, and suppose we are given a morphism $\Phi \colon C(V) \to C(V)$ with $\Phi_{(1)}=\mathrm{id}_V$, and a degree $1$ $(F' \circ \Phi,F)$-coderivation $H$ with $H_{(1)}=0$, such that the restrictions of $\Phi$ and $H$ to $C^{\leq(n-1)}(V)$ satisfy the equations
	$$ \begin{cases} \mu' \circ \Phi - \Phi \circ \mu = 0 \\ F' \circ \Phi - F = \nu \circ H + H \circ \mu.\end{cases} $$
	One easily checks that the following two equations are satisfied:
	$$ \mu' \circ (\mu' \circ \Phi - \Phi \circ \mu) + (\mu' \circ \Phi - \Phi \circ \mu) \circ \mu = 0$$
	and
	$$ F \circ (\mu' \circ \Phi - \Phi \circ \mu) = \nu \circ (F' \circ \Phi - F - \nu \circ H - H \circ \mu) -(F' \circ \Phi - F - \nu \circ H - H \circ \mu) \circ \mu. $$
	Considering the arity $n$ term of these equations and using the fact that both $\mu' \circ \Phi - \Phi \circ \mu = 0$ and $F' \circ \Phi - F - \nu \circ H - H \circ \mu = 0$ in arities below $n$, we deduce that
	$$ \partial (\mu' \circ \Phi - \Phi \circ \mu)_{(n)} = 0$$
	and
	$$ f \circ (\mu' \circ \Phi - \Phi \circ \mu)_{(n)} = \partial \Big( (F' \circ \Phi)_{(n)} - F_{(n)} - (\nu \circ H + H \circ \mu)_{(n)} \Big).$$
	As in the previous proofs it follows that $(\mu' \circ \Phi - \Phi \circ \mu)_{(n)}$ is itself a boundary, say $\partial e$. Then if we let $\Phi'$ denote the morphism which has $\Phi'_{(n)} = \Phi_{(n)} - e$ and agrees with $\Phi$ in all other arities, we will have $ (\mu' \circ \Phi' - \Phi' \circ \mu)_{(n)} = 0$. Hence also
	$$ \partial \Big( (F' \circ \Phi')_{(n)} - F_{(n)} - (\nu \circ H + H \circ \mu)_{(n)}\Big) = 0,$$
	and so (again using that $f$ induces a quasi-isomorphism  $ \underline\Hom_R(C^n(V),V) \to \underline\Hom_R(C^n(V),W)$) we see that $(F' \circ \Phi')_{(n)} - F_{(n)} - (\nu \circ H + H \circ \mu)_{(n)}$ is the sum of a boundary and the image of a cycle under $f$, say $f \circ e' + \partial e''$. We now modify $\Phi'$ and $H$ in arity $n$, setting $\Phi''_{(n)} = \Phi'_{(n)} - e'$ and $H'_{(n)} = H_{(n)}-e''$. Now $(F' \circ \Phi'')_{(n)} - F_{(n)} - (\nu \circ H' + H' \circ \mu)_{(n)} = 0$ and the proof is done by induction.  
	
	If $C$ is a symmetric operad, a nearly identical argument would show that there exists a coderivation $H$ in the usual sense for which $$ \begin{cases} \mu' \circ \Phi - \Phi \circ \mu = 0 \\ F' \circ \Phi - F = \nu \circ H + H \circ \mu.\end{cases} $$
	In particular, $\Phi$ still provides an isomorphism between the two transferred structures $\mu$ and $\mu'$, but $H$ can no longer be interpreted as a homotopy between $\infty$-morphisms.\end{proof}

The proof of \cref{thmD} very similar to the three preceding proofs, and it is obtained by modifying the proof of \cref{thmC} in exactly the same way as \cref{thmB} is obtained by modifying the proof of \cref{thmA}. We omit the argument. 

The notion of homotopy between morphisms of coalgebras over a symmetric cooperad is a bit more complicated to define. There are three possible definitions, all of which give rise to equivalent notions, but not obviously so; see \cite{dotsenkoponcin}. None of the possible definitions make sense in general unless the ground ring $R$ contains $\mathbf Q$, and it is not clear to me whether the above argument could be modified to produce a homotopy in this symmetric sense, for any of the definitions.

\section*{Acknowledgements}
The author gratefully acknowledges support by ERC-2017-STG 759082 and a Wallenberg Academy Fellowship.


\end{document}